\documentclass[10pt]{article}
\usepackage[margin= 1.25in]{geometry}

\usepackage{amsmath}
\usepackage{amssymb}
\usepackage{amsthm}
\usepackage{dsfont}
\usepackage[T1]{fontenc}
\usepackage[mathscr]{eucal}
\usepackage[usenames,dvipsnames]{color}
\usepackage{graphicx}
\usepackage{mathrsfs}
\usepackage{nameref}
\usepackage{pgfplots}
\usepackage{relsize}
\usepackage{setspace}
\usepackage{textcomp}
\usepackage{tikz}
	\usetikzlibrary{matrix}
\usepackage[normalem]{ulem}
\usepackage{verbatim}

\numberwithin{equation}{section}

\usepackage{stmaryrd}

\def\polhk#1{\setbox0=\hbox{#1}{\ooalign{\hidewidth\lower1.5ex\hbox{`}\hidewidth\crcr\unhbox0}}}

\def\ba{\begin{aligned}}
\def\ea{\end{aligned}}
\def\be{\begin{equation}}
\def\ee{\end{equation}}
\def\beu{\begin{equation*}}
\def\eeu{\end{equation*}}

\def\C{\mathbb{C}}

\def\R{\mathbb{R}}

\def\S{\mathbb{S}}

\def\Z{\mathbb{Z}}

\def\!={\neq}

\theoremstyle{definition}

\newtheorem{defn}[equation]{Definition}
\newtheorem{lem}[equation]{Lemma}
\newtheorem{prop}[equation]{Proposition}
\newtheorem{thm}[equation]{Theorem}

\newtheorem{ques}[equation]{Question}

\newtheorem{cor}[equation]{Corollary}
\newtheorem{exam}[equation]{Example}
\newtheorem{conj}[equation]{Conjecture}
\newtheorem{remark}[equation]{Remark}

\begin{document}

\title{\Large Triviality of Equivariant Maps in Crossed Products and Matrix Algebras}
\author{Benjamin Passer\thanks{Partially supported by a Zuckerman Fellowship at the Technion, the Simons Foundation grant 346300, and the Polish Government MNiSW 2015-2019 matching fund. This work is part of a project sponsored by EU grant H2020-MSCA-RISE-2015-691246-QUANTUM DYNAMICS.} \\ Technion-Israel Institute of Technology \\ Haifa, Israel \\ \\ benjaminpas@technion.ac.il}

\date{}

\maketitle

\begin{abstract}
We consider a \lq\lq twisted\rq\rq\hspace{0pt} noncommutative join procedure for unital $C^*$-algebras which admit actions by a compact abelian group $G$ and its discrete abelian dual $\Gamma$, so that we may investigate an analogue of Baum-D\polhk{a}browski-Hajac noncommutative Borsuk-Ulam theory in the twisted setting. Namely, under what conditions is it guaranteed that an equivariant map $\phi$ from a unital $C^*$-algebra $A$ to the twisted join of $A$ and $C^*(\Gamma)$ cannot exist? This pursuit is motivated by the twisted analogues of even spheres, which admit the same $K_0$ groups as even spheres and have an analogous Borsuk-Ulam theorem that is detected by $K_0$, despite the fact that the objects are not themselves deformations of a sphere. We find multiple sufficient conditions for twisted Borsuk-Ulam theorems to hold, one of which is the addition of another equivariance condition on $\phi$ that corresponds to the choice of twist. However, we also find multiple examples of equivariant maps $\phi$ that exist even under fairly restrictive assumptions. Finally, we consider an extension of unital contractibility (in the sense of D\polhk{a}browski-Hajac-Neshveyev) \lq\lq modulo $k$.\rq\rq\hspace{0pt}
\end{abstract}

\vspace{.2 in}


\vspace{.2 in}


\vspace{.2 in}

\section{Introduction}\label{sec:intro}

If $X$ and $Y$ are compact Hausdorff spaces which admit free actions of a nontrivial finite group $G$, then continuous, equivariant maps from $X$ to $Y$ are severely restricted by dimension and homotopy invariants. Such claims are rooted in the Borsuk-Ulam theorem, which concerns spheres $\mathbb{S}^k$ and the free $\mathbb{Z}/2\mathbb{Z}$ action $x \mapsto -x$. In this setting, an equivariant map $f$ is just an \textit{odd} function, that is, a function that satisfies $f(-x) = -f(x)$ for each $x$ in the domain.

\begin{thm}[Borsuk-Ulam Theorem]
If $n$ is a positive integer, then the following hold.
\begin{enumerate}
\item Every continuous function from $\S^n$ to $\R^n$ admits at least one pair of opposite points $x$ and $-x$ with the same image.
\item Every continuous, odd function from $\S^n$ to $\R^n$ has a zero.
\item There is no continuous, odd function from $\S^n$ to $\S^{n-1}$.
\item Every continuous, odd function from $\S^{n-1}$ to $\S^{n-1}$ is homotopically nontrivial.
\item Every continuous, odd function from $\S^{n-1}$ to $\S^{n-1}$ has odd degree.
\end{enumerate}
\end{thm}

Items 1-4 above are equivalent, and they follow from the stronger item 5. Generalizations and consequences of the Borsuk-Ulam theorem abound in combinatorics and algebraic topology, in which more general spaces, groups, and actions replace spheres, $\mathbb{Z}/2\mathbb{Z}$, and the antipodal action. More recently, significant progress has been made in generalizing these results to noncommutative topology, opening up many new potential problems. The reason for this is twofold: first, the behavior of noncommutative $C^*$-algebras can be quite different than the picture painted by compact Hausdorff topology, and second, the noncommutative setting allows discussion of compact \textit{quantum} group actions in addition to actions of compact groups. The aim of such results \cite{yam13, tag12, bau15, self16thetasphere, dab16, selfsaturatedPUBLISHED, self17antisphere, selfjoinACCEPTED, dab16revisitedARXIV} is to restrict the existence or homotopy properties of equivariant homomorphisms between unital $C^*$-algebras.  As such, the noncommutative join constructions in \cite{dab15} play a major role in this story.

\begin{defn}[\cite{dab15}]\label{defn:joins}
Let $A$ and $B$ be unital $C^*$-algebras. Then the \textit{noncommutative join} of $A$ and $B$ is 
\beu
A \circledast B = \{f \in C([0, 1], A \otimes_\mathrm{min} B): f(0) \in \C \otimes B, f(1) \in A \otimes \C \}.
\eeu
\noindent If $(H, \Delta)$ is a compact quantum group and $\delta: A \to A \otimes H$ is a free coaction of $H$ on $A$, then the \textit{equivariant noncommutative join} of $A$ and $H$ is
\beu
A \circledast_\delta H = \{f \in C([0, 1], A \otimes_\mathrm{min} H): f(0) \in \C \otimes H, f(1) \in \delta(A)\}.
\eeu
\noindent Further, $A \circledast_\delta H$ admits a free coaction $\delta_\Delta$, which applies $\mathrm{id} \otimes \Delta$ pointwise on $A \otimes_\mathrm{min} H$.
\end{defn}

In the above definition, freeness is meant in the sense of \cite{ell00} (see also \cite{bau13PUBLISHED}). When $H$ is equal to $C(G)$ for a compact group $G$, the two joins $A \circledast C(G)$ and $A \circledast_\delta C(G)$ are isomorphic, and if $A \circledast C(G)$ is given the diagonal action of $G$, the natural choice of isomorphism is equivariant for the diagonal action and $\delta_\Delta$. That is, the equivariant join is needed precisely when the quantum group $H$ is not commutative, to avoid applying a diagonal action. Further, the iterated join $\mathlarger{\mathlarger{\mathlarger{\circledast}}}_{i=1}^n C(\mathbb{Z}/2\mathbb{Z})$ is isomorphic to $C(\mathbb{S}^{n-1})$, and the induced diagonal action (at each additional join) is just the antipodal action of $\mathbb{Z}/2\mathbb{Z}$ on $C(\mathbb{S}^{n-1})$. The conjectures of Baum, D\polhk{a}browski, and Hajac in \cite{bau15} therefore directly generalize the Borsuk-Ulam theorem.

\begin{conj}(\cite[Conjecture 2.3]{bau15})\label{conj:NCBU-BDH}
Let $A$ be a unital $C^*$-algebra with a free action $\delta$ of a nontrivial compact quantum group $(H, \Delta)$. Also, let $A \circledast_\delta H$ denote the equivariant noncommutative join of $A$ and $H$, with the induced action of $(H, \Delta)$ given by $\delta_\Delta$. Then both of the following hold.
\begin{enumerate}
\item There does not exist a $(\delta, \delta_\Delta)$-equivariant, unital $*$-homomorphism from $A$ to $A \circledast_\delta H$.
\item There does not exist a $(\Delta, \delta_\Delta)$-equivariant, unital $*$-homomorphism from $H$ to $A \circledast_\delta H$.
\end{enumerate}
\end{conj}

We note that when $H = C(\mathbb{Z}/2\mathbb{Z})$, $A \circledast_\delta C(\mathbb{Z}/2\mathbb{Z})$ is equivariantly isomorphic to the\textit{ unreduced suspension} of $A$, $\Sigma A$, discussed in \cite{dab16}. More specifically, $\Sigma A$ is the (non-equivariant) noncommutative join $A \circledast C(\mathbb{Z}/2\mathbb{Z})$, and the action presented in \cite{dab16} is the diagonal action of $\mathbb{Z}/2\mathbb{Z}$.

Conjecture \ref{conj:NCBU-BDH} type 1 holds when $H$ is a compact quantum group with a torsion character other than a counit, as in \cite[Corollary 2.7]{dab16revisitedARXIV}, from a direct application of the case $H = C(\Z/k\Z)$ in \cite[Corollary 2.4]{selfsaturatedPUBLISHED}. Conjecture 2 is false, and counterexamples exist for compact groups acting on nuclear $C^*$-algebras from \cite[Theorem 2.6]{selfjoinACCEPTED}. However, there are certain key examples for which the type 2 conjecture holds, as in \cite{selfjoinACCEPTED, haj15ARXIV, bau15, dab16revisitedARXIV}. In \cite{selfsaturatedPUBLISHED}, we proposed a different type of join (and similarly, unreduced suspension) for $C^*$-algebras with free actions of $\mathbb{Z}/k\mathbb{Z}$, replacing the tensor product with a crossed product. We generalize this definition and adopt new terminology/notation to avoid confusion with Definition \ref{defn:joins}.

If $\Gamma$ is a discrete abelian group and $G$ is its compact abelian Pontryagin dual, then an action $\alpha$ of $G$ on $A$ is equivalent to a grading of $A$ by $\Gamma$, or a coaction $\delta: A \to A \otimes C^*(\Gamma)$ of the compact quantum group $C^*(\Gamma) = C(G)$. For $\gamma \in \Gamma$, which we may view as a character on $G$, the spectral subspace
\beu 
A_\gamma = \{a \in A: \alpha_g(a) = \gamma(g)a \text{ for all } g \in G\}
\eeu
is written in coaction form as the $\gamma$\textit{-isotypic subspace}
\beu
A_\gamma = \{a \in A: \delta(a) = a \otimes \gamma\},
\eeu
and the subspaces $A_\gamma$ produces a grading of $A$. When $\gamma = 1$, $A_\gamma$ is called the \textit{fixed-point subalgebra}. Moreover, freeness of the (co)action is equivalent to a saturation property
\be\label{eq:saturationfree}
\forall \gamma \in \Gamma, \hspace{.25 in} 1 \in \overline{A_\gamma A_\gamma^*}.
\ee
See \cite[Theorem 0.4]{bau13PUBLISHED} for the equivalence of freeness and saturation in greater generality, as well as \cite{phi09,kis80,goo94} for some useful discussion of saturation properties in the context of group actions, and \cite{mon93} for related conditions when $\Gamma$ is finite (but not necessarily abelian).

\begin{defn}\label{defn:twistedjoin}
Let $\Gamma$ be a discrete abelian group with $G = \widehat{\Gamma}$. Suppose $A$ is a unital $C^*$-algebra and $\beta$ is an action of $\Gamma$ on $A$. Then the $\beta$-\textit{twisted join} of $A$ and $C^*(\Gamma) \cong C(G)$ is
\be\label{eq:twistedjoindefeq}
J(A, \beta) = \{f \in C([0, 1], A \rtimes_\beta \Gamma): f(0) \in C^*(\Gamma), f(1) \in A\}.
\ee
When $\Gamma = \Z/2\Z$, we call $J(A, \beta)$ the $\beta$-\textit{twisted unreduced suspension of} $A$ \textit{and} $C^*(\Z/2\Z) \cong C(\Z/2\Z)$. If $\beta$ is the trivial action, then $J(A, \beta)$ is just $A \circledast C(G)$.
\end{defn}

There is a natural grading of $A \rtimes_\beta \Gamma$ by $\Gamma$, which extends to a grading of $J(A, \beta)$ pointwise, that places each $\gamma \in \Gamma$ in the $\gamma$-isotypic subspace and each $a \in A$ in the fixed-point subalgebra. This grading is determined from the action $\widehat{\beta}$ of the compact group $G = \widehat{\Gamma}$ on $A \rtimes_\beta \Gamma$, given by
\beu
\widehat{\beta}_g(a) = a, a \in A \hspace{1 in} \widehat{\beta}_g(\gamma) = g(\gamma), \gamma \in \Gamma.
\eeu
If $\alpha$ is an action of $G$ on $A$ which commutes with $\beta$, then we extend $\alpha$ to $A \rtimes_\beta \Gamma$ so that $\alpha$ fixes all elements of the group $\Gamma$. The composition $\alpha \widehat{\beta} = \widehat{\beta} \alpha$ produces an action of $G$ on $A \rtimes_\beta \Gamma$, and hence a grading of $A \rtimes_\beta \Gamma$ by $C^*(\Gamma)$, generated by the following rule: if $a$ is in the $\tau$-isotypic subspace of $A$ from $\alpha$, and $\gamma \in \Gamma$, then $a\gamma = \gamma \beta(a)$ is in the $\tau \gamma$-isotypic subspace of $A \rtimes_\beta \Gamma$. In particular, if $\Gamma = \Z/2\Z$ is generated by $\mu$, then $\alpha \widehat{\beta}$ is generated by the automorphism
\beu
a_0 + a_1 \mu \hspace{.1 in} \mapsto \hspace{.1 in} \alpha(a_0) - \alpha(a_1) \mu.
\eeu
Similarly, if $\Gamma = \Z/k\Z$ is generated by $\mu$ and $\omega = e^{2 \pi i / k}$, then $\alpha \widehat{\beta}$ is generated by the automorphism
\beu
a_0 + a_1 \mu + \ldots + a_{k-1} \mu^{k-1} \hspace{.1 in} \mapsto \hspace{.1 in} \alpha(a_0) + \omega \alpha(a_1) \mu + \ldots + \omega^{k-1} \alpha(a_{k-1}) \mu^{k-1}.
\eeu

\begin{defn}
If $\alpha$ is an action of a compact abelian group $G$ on $A$ that commutes with an action $\beta$ of $\Gamma = \widehat{G}$ on $A$, then let $\widetilde{\alpha}$ denote the pointwise application of the action $\alpha \widehat{\beta}$ on $J(A, \beta)$.
\end{defn}

In the non-twisted setting, when $\beta$ is trivial, $\widetilde{\alpha}$ corresponds to the diagonal action of $G$ on $A \circledast C(G)$. Hence, if $G = \Z/k\Z$, $\alpha$ is free, and $\beta$ is trivial, then there are no $(\alpha, \widetilde{\alpha})$-equivariant unital $*$-homomorphisms from $A$ to $J(A, \beta)$, as seen in \cite[Corollary 2.4]{selfsaturatedPUBLISHED}. To extend nonexistence results about equivariant maps $A \to J(A, \beta)$ to cases where $\beta$ is nontrivial, we know from \cite[Example 3.7]{selfsaturatedPUBLISHED} that some assumption on $\beta$, or the equivariant map in question, is still necessary. Here we consider two assumptions that insist $\beta$ is not too far removed from the trivial action, in an attempt to generalize known examples and theorems.

\begin{ques}\label{ques:thequestion}
Suppose $\alpha$ and $\beta$ are commuting actions of $\Z/k\Z$ on a unital $C^*$-algebra $A$, and $\alpha$ is free. Consider the following conditions on $\beta$.
\begin{enumerate}
\item The action $\beta$ is not free.
\item The individual automorphisms of $\beta$ are connected within $\mathrm{Aut}(A)$ to the trivial automorphism.
\end{enumerate}
Is either condition sufficient to guarantee that there are no $(\alpha, \widetilde{\alpha})$-equivariant, unital $*$-homomorphisms from $A$ to $J(A, \beta)$?
\end{ques}

The conditions were determined through the computation of various examples, as in \cite{selfsaturatedPUBLISHED, self17antisphere}, chief among them odd-dimensional $\theta$-deformed spheres (defined in \cite{mat91, nat97}) and twisted versions thereof. In section \ref{sec:nonexistence} we consider more restrictive assumptions that guarantee nonexistence of equivariant maps from $A$ to $J(A, \beta)$. However, neither condition of Question \ref{ques:thequestion} is actually sufficient in general, as shown in Theorems \ref{thm:condition1fails} and \ref{thm:condition2fails}.  Finally, in the remainder of section \ref{sec:existence} we apply the usual embedding $A \rtimes_\beta \Z/k\Z \hookrightarrow M_k(A)$ to see consequences of the aforementioned results for deforming the diagonal inclusion $A \hookrightarrow M_k(A)$ to finite-dimensional and one-dimensional representations. Extending the definition of unital contractibility to unital contractibility \lq\lq modulo $k$\rq\rq\hspace{0pt}, we find that the connection between contractibility properties and examples of equivariant maps $A \to J(A, \beta)$ is not as direct as in \cite[Corollary 2.7]{dab16revisitedARXIV}.


\section{Nonexistence}\label{sec:nonexistence}

The noncommutative Borsuk-Ulam theorem in \cite[Corollary 2.4]{selfsaturatedPUBLISHED} can be proved using an iteration procedure, which is not as immediate in the twisted setting. Specifically, while morphisms $A \to B$ induce morphisms $A \circledast C(G) \to B \circledast C(G)$, the same does not generally hold for twisted joins without additional equivariance requirements. In this section, we consider sufficient conditions that guarantee there are no equivariant morphisms $A \to J(A, \beta)$. To avoid unnecessary repretition, $\Gamma$ will always refer to a discrete abelian group, and $G$ will be its compact abelian Pontryagin dual group $\widehat{\Gamma}$.

If $\beta$ is an action of $\Gamma$ on $A$, then we may extend $\beta$ to $A \rtimes_\beta \Gamma$ so that each group element is in the fixed-point subalgebra. We will also refer to this action, as well as its pointwise application on $J(A, \beta)$, by $\beta$.

\begin{lem}\label{lem:equivariantiteration}
Suppose $\beta_A$ and $\beta_B$ are actions of $\Gamma$ on unital $C^*$-algebras $A$ and $B$. If $\phi: A \to B$ is a $(\beta_A, \beta_B)$-equivariant, unital $*$-homomorphism, then the rule
\be \label{eq:pointwisephi}
a \in A \mapsto \phi(a) \hspace{1 in} \gamma \in \Gamma \mapsto \gamma
\ee
produces a homomorphism $A \rtimes_{\beta_A} \Gamma \to B \rtimes_{\beta_B} \Gamma$. If $J_\phi: J(A, \beta_A) \to J(B, \beta_B)$ denotes the pointwise application of this map, then $J_\phi$ is $(\beta_A, \beta_B)$-equivariant.

If, in addition, $\alpha_A$ and $\alpha_B$ are actions of $G = \widehat{\Gamma}$ on $A$ and $B$ which commute with $\beta_A$ and $\beta_B$, respectively, and $\phi$ is also $(\alpha_A, \alpha_B)$-equivariant, then $J_\phi$ is pointwise $(\widetilde{\alpha_A}, \widetilde{\alpha_B})$-equivariant.
\end{lem}
\begin{proof}
The function $\psi: A \rtimes_{\beta_A} \Gamma \to B \rtimes_{\beta_B} \Gamma$ defined by (\ref{eq:pointwisephi})
is a unital $*$-homomorphism by the universal property of crossed products. The associated homomorphism
\beu
\mathrm{id} \otimes \psi: C([0,1]) \otimes \left(A \rtimes_{\beta_A} \Gamma \right) \mapsto C([0,1]) \otimes \left(B \rtimes_{\beta_B} \Gamma\right)
\eeu
respects the boundary conditions of the twisted join, as $\psi(C^*(\Gamma)) = C^*(\Gamma)$ and $\psi(A) \subset B$. Therefore, it induces a homomorphism $J_\phi: J(A, \beta_A) \to J(B, \beta_B)$, which is pointwise $(\beta_A, \beta_B)$-equivariant by design. Finally, the actions $\widetilde{\alpha_A}$ and $\widetilde{\alpha_B}$ are pointwise applications of $\alpha_A \widehat{\beta_A}$ and $\alpha_B \widehat{\beta_B}$, so it suffices to prove that $\psi$ is $(\alpha_A \widehat{\beta_A}, \alpha_B \widehat{\beta_B})$-equivariant. For $g \in G$,

\beu \ba
(\alpha_B \widehat{\beta_B})_g\left(\psi\left( \sum a_\gamma \cdot \gamma \right)\right)	&= (\alpha_B \widehat{\beta_B})_g\left( \sum \phi(a_\gamma) \cdot \gamma \right)  \hspace{1 in} \\
	&= \sum (\alpha_B)_g( \phi(a_\gamma)) \cdot g(\gamma) \\
	&= \sum \phi((\alpha_A)_g(a_\gamma)) \cdot g(\gamma) \\ 
	&= \psi\left( \sum (\alpha_A)_g(a_\gamma) \cdot g(\gamma)\right) \\ 
	&= \psi\left( (\alpha_A \widehat{\beta_A})_g \left( \sum a_\gamma \cdot \gamma \right) \right),
\ea \eeu
so $\psi$ is $(\alpha_A \widehat{\beta_A}, \alpha_B \widehat{\beta_B})$-equivariant, and $J_\phi$ is $(\widetilde{\alpha_A}, \widetilde{\alpha_B})$-equivariant.
\end{proof}

The additional equivariance demanded in Lemma \ref{lem:equivariantiteration} is automatic if $\beta$ is trivial. It follows that the result below generalizes \cite[Corollary 2.4]{selfsaturatedPUBLISHED}.

\begin{thm}\label{thm:doubleequiv}
Let $\Gamma = G = \Z/k\Z$, $k \geq 2$, and suppose $A$ is a unital $C^*$-algebra with two commuting actions $\alpha$ and $\beta$ of $\Z/k\Z$, where $\alpha$ is free. Then there is no unital $*$-homomorphism $\phi: A \to J(A, \beta)$ which is both $(\alpha, \widetilde{\alpha})$-equivariant and $(\beta, \beta)$-equivariant.
\end{thm}
\begin{proof}
We let $A_0 := A$, which admits actions $\alpha_0 := \alpha$ and $\beta$. The twisted join $A_1 := J(A, \beta)$ admits the pointwise action of $\beta$ (still denoted $\beta$), and we let $\alpha_1 = \widetilde{\alpha_0}$. Iterating this procedure, we define $A_n$ as an iterated twisted join of $A$ via the rule
\beu
A_n = J(A_{n-1}, \beta), \hspace{.5 in} \alpha_n = \widetilde{\alpha_{n-1}}.
\eeu
Suppose $\phi_0 := \phi: A_0 \to A_1$ is both $(\alpha_0, \alpha_1)$-equivariant and $(\beta, \beta)$ equivariant. Then repeated applications of Lemma \ref{lem:equivariantiteration} produce $(\alpha_{n-1}, \alpha_n)$-equivariant and $(\beta, \beta)$-equivariant maps $\phi_n: A_{n-1} \to A_n$. From composition of these maps in a chain, we find that for any $n \in \Z^+$, there is an $(\alpha_0, \alpha_n)$-equivariant, $(\beta, \beta)$-equivariant map $\Phi_n: A_0 \to A_n$.

Next, we apply a similar iteration procedure for maps into $C^*(\Z/k\Z)$. Let $B_0 := C^*(\Z/k\Z)$ be graded by $\Z/k\Z$ in the usual way, inducing an action $\rho_0$. Since $\phi_0: A_0 \to J(A_0, \beta)$ is $(\alpha, \widetilde{\alpha})$-equivariant and $(\beta, \beta)$-equivariant, evaluation at the $t = 0$ endpoint of (\ref{eq:twistedjoindefeq}) shows that there is a map $\psi_0: A_0 \to B_0$ which is $(\alpha_0, \rho_0)$-equivariant and $(\beta, \textrm{triv})$-equivariant. Define
\beu
B_n = J(B_{n-1}, \textrm{triv}), \hspace{.5 in} \rho_n = \widetilde{\rho_{n-1}},
\eeu
and note that iteration of Lemma \ref{lem:equivariantiteration} once more establishes that for each $n$, there is a $(\alpha_n, \rho_n)$-equivariant, $(\beta, \textrm{triv})$-equivariant map $\psi_n: A_n \to B_n$. Since $B_n$ is defined using a trivial twist in its iterated join procedure, we note that $B_n \cong \mathlarger{\mathlarger{\mathlarger{\circledast}}}_{i=1}^{n+1} C(\Z/k\Z) = C((\Z/k\Z)^{*n+1})$, and $\rho_n$ is the diagonal action.

Finally, the composition $\psi_n \circ \Phi_n: A \to C(\Z/k\Z^{*n})$ is $(\alpha, \text{diag})$-equivariant. Since the action $\alpha$ is free, fix a generator $\gamma \in \Z/k\Z$ and note that by (\ref{eq:saturationfree}), $1 \in \overline{A_\gamma A_\gamma^*}$. In particular, there is a finite $N$ such that there exist $a_1, b_1, \ldots, a_N, b_N \in A_\gamma$ with $\sum a_i b_i^*$ invertible. On the other hand, from the increasing connectivity of the iterated joins of $\Z/k\Z$, the number of elements in the $\gamma$-isotypic subspace of $C(\Z/k\Z^{*n})$ required to generate an invertible grows without bound as $n$ increases (see \cite{dol83} and \cite[Definition 4.3.1 and Proposition 4.4.3]{mat03}). It follows that for large $n$, the existence of the $(\alpha, \text{diag})$-equivariant map $\psi_n \circ \Phi_n: A \to C((\Z/k\Z)^{*n})$ leads to a contradiction.
\end{proof} 

Theorem \ref{thm:doubleequiv} does not assume any condition on the freeness or non-freeness of $\beta$. However, we note that the pointwise action $\beta$ on $J(A, \beta)$ is not free, as at the endpoint $t = 0$ of (\ref{eq:twistedjoindefeq}), $\beta$ corresponds to the trivial action on $C^*(\Z/k\Z)$. Therefore, if the original action $\beta$ on $A$ is free, there cannot be a $(\beta, \beta)$-equivariant unital $*$-homomorphism from $A$ to $J(A, \beta)$, regardless of any other action $\alpha$. It follows that Theorem \ref{thm:doubleequiv} is useful when $\beta$ is \textit{not} free, so it may be safely viewed inside the framework of Question \ref{ques:thequestion}, condition 1.

If $A = C(X)$ and $k \in \Z^+$ is prime, then any non-free action of $\Z/k\Z$ has a nonempty fixed point set. When this set is an equivariant retract of $X$, where $X$ is acted upon freely, we may produce a twisted Borsuk-Ulam theorem as follows.

\begin{prop}
Let $k \in \Z^+$ be prime, and let $X$ be a compact Hausdorff space with two commuting $\Z/k\Z$ actions $\alpha$ and $\beta$, where $\alpha$ is free and $\beta$ is not. Let $Y = \mathrm{Fix}(\beta) \not= \varnothing$ be equipped with the restricted action $\gamma = \alpha|_Y$, and suppose there is an $(\alpha, \gamma)$-equivariant continuous function $f: X \to Y$. Then there is no unital, $(\alpha, \widetilde{\alpha})$-equivariant $*$-homomorphism from $C(X)$ to $J(C(X), \beta)$.
\end{prop}
\begin{proof}
Suppose $\phi: C(X) \to J(C(X), \beta)$ is $(\alpha, \widetilde{\alpha})$-equivariant. First, we have that the dual map $f^*: C(Y) \to C(X)$ to $f: X \to Y$ is $(\gamma, \alpha)$-equivariant. Second, the restriction map $C(X) \to C(Y)$ is certainly $(\alpha, \gamma)$-equivariant and $(\beta, \mathrm{triv})$-equivariant, so it may be applied pointwise on the twisted joins by Lemma \ref{lem:equivariantiteration}. This gives
\beu
C(Y) \xrightarrow[(\gamma, \alpha)]{f^*} C(X) \xrightarrow[(\alpha, \widetilde{\alpha})]{\phi} J(C(X), \beta) \xrightarrow[(\widetilde{\alpha}, \widetilde{\gamma})]{\textrm{pointwise} \hspace{2 pt} |_Y} J(C(Y), \mathrm{triv}) \xrightarrow[(\widetilde{\gamma}, \textrm{diag})]{\cong} C(Y) \circledast C(\Z/k\Z).
\eeu
The composition is $(\gamma, \textrm{diag})$-equivariant, contradicting \cite[Corollary 2.4]{selfsaturatedPUBLISHED} (or, rather, the topological result \cite[Corollary 3.1]{vol05}).
\end{proof}

The following proposition is motivated by a topological picture: if $A = C(X)$ is commutative and $\alpha$ is a free action of $\Z/k\Z$ which permutes the components of $X$, then for non-free actions $\beta$ on $X$, there might not exist equivariant maps $A \to J(A, \beta)$ due to the existence of finite-order unitaries. The arguments depend upon a standard matrix expansion map
\be\label{eq:expansionmapeq}
E: A \rtimes_\beta \Z/k\Z \to M_k(A),
\ee
defined by mapping $a \in A$ to the diagonal matrix with entries $a, \beta_1(a), \ldots, \beta_{k-1}(a)$ and mapping a generator $\mu \in \Z/k\Z$ to a $\{0,1\}$-valued matrix $S$ which induces the shift $e_n \mapsto e_{n+1}$ in the standard basis of $\C^k$. The action $\beta$ on $A \rtimes_\beta \Z/k\Z$ then corresponds to the entrywise application of $\beta$ on the subalgebra $E(A \rtimes_\beta \Z/k\Z) \subseteq M_k(A)$.

\begin{prop}\label{prop:connectedness}
Fix $k \geq 2$ and a generator $\mu$ of $\widehat{\Z/k\Z} \cong \Z/k\Z$. Suppose $A$ is a unital $C^*$-algebra with a free action $\alpha$ of $\Z/k\Z$, such that there is a unitary $x$ in the $\mu$-isotypic subspace with $x^k = 1$. Further, let $\beta$ be an action of $\Z/k\Z$ on $A$ which commutes with $\alpha$, such that the ideal $I$ generated by terms $\beta_m(a) - a$, $m \in \Z/k\Z$, $a \in A$ is proper, and the $K_0(A/I)$-class of the unit $1$ is not divisible by $k$. Then there is no $(\alpha, \widetilde{\alpha})$-equivariant unital $*$-homomorphism from $A$ to $J(A, \beta)$.
\end{prop}
\begin{proof}
If $\phi: A \to J(A, \beta)$ is $(\alpha, \widetilde{\alpha})$-equivariant, then $\phi(x)$ determines a path of unitaries $u_t \in A \rtimes_\beta \Z/k\Z$ with $u_t^k = 1$, connecting $u_0 \in C^*(\Z/k\Z)$ to $u_1 \in A$. Now, $u_t$ is also in the $\mu$-isotypic subspace of $A \rtimes_\beta \Z/k\Z$ for $\alpha \widehat{\beta}$, so $u_0$ is of the form $c \mu$ for some $c \in \C$ with $c^k = 1$, and $u_1$ is in the $\mu$-isotypic subspace of $A$ for $\alpha$. Let $v_t = c^{-1} u_t$ and define the projections
\beu
P_t = \cfrac{1+v_t+v_t^2+\ldots+v_t^{k-1}}{k}
\eeu
for $t \in [0,1]$. Apply the expansion map $E$ of (\ref{eq:expansionmapeq}) and the entrywise quotient $A \to A/I$, where $I$ identifies any $a \in A$ with $\beta_m(a)$ for any $m$. Then the quotient of $E(P_t)$ produces a path of projections in $M_k(A/I)$ connecting a matrix $T \in M_k(\C)$ to a diagonal matrix with entries $a + I, a + I, \ldots, a + I$ for some $a \in A$. The matrix $T$ is of the form $\cfrac{1}{k} \sum\limits_{n=0}^{k-1} S^k$ for an order $k$ shift $S \in U_k(\C)$. Since $S$ has eigenvalues $1, \omega, \ldots, \omega^{k-1}$ for a primitive $k$th root of unity $\omega$, it follows that $T$ has rank $1$, and $T$ is equivalent in $K_0(A/I)$ to the unit $1$. On the other hand, the diagonal matrix with entries $a + I, \ldots, a + I$ is equivalent in $K_0(A/I)$ to the sum of $a+I$ with itself $k$ times. Therefore, $1 \in K_0(A/I)$ is divisible by $k$.
\end{proof}

In Proposition \ref{prop:connectedness}, the assumptions imply that $\alpha$ is free and $\beta$ is not free. Below we write a commutative subcase of the result.

\begin{cor}
Fix $k \geq 2$ and let $X$ be a disconnected compact Hausdorff space. Let $\alpha$ and $\beta$ be two commuting actions of $\Z/k\Z$ on $X$ such that $\text{Fix}(\beta)$ is nonempty and there is a clopen set $Y \subseteq X$ such that $Y \cup \alpha_1(Y) \cup \ldots \cup \alpha_{k-1}(Y) = X$ and the union is disjoint. Then there is no $(\alpha, \widetilde{\alpha})$-equivariant unital $*$-homorphism from $C(X)$ to $J(C(X), \beta)$.
\end{cor}
\begin{proof}
Fix a primitive $k$th root of unity $\mu$, and define a function $f$ on $X$ by assigning $f(p) = \mu^m$ if and only if $y \in \alpha^m(Y)$. Then $f \in C(X)$ is a unitary in the $\mu$-isotypic subspace of $\alpha$, such that $f^k = 1$. Since $Z := \text{Fix}(\beta)$ is not empty, the ideal in $C(X)$ generated by $\beta_{m}(a) - a$, $0 \leq m \leq k -1$, $a \in C(X)$ is proper. Specifically, $C(X) / I \cong C(Z)$. Next, there is a map $K_0(C(Z)) \to K_0(\C) \cong \Z$ induced by evaluation at a point $q \in Z$, which sends the unit $1 \in C(Z)$ to $1$, so $1$ is indivisible by $k$ in $K_0(C(Z))$. Finally, we may apply Proposition \ref{prop:connectedness}.
\end{proof}

If $k$ is prime, note that $\text{Fix}(\beta)$ is empty if and only if $\beta$ is free. Therefore, these results also fall under the umbrella of Question \ref{ques:thequestion}, condition $1$.  Finally, following the $K$-theoretic computations in \cite{self16thetasphere, self17antisphere}, we see that $\theta$-deformed spheres, and certain twisted unreduced suspensions thereof, admit twisted Borsuk-Ulam theorems. First, we recall the definitions, as in \cite{mat91, nat97, self17antisphere}.

\begin{defn}\label{defn:spheres}
For an antisymmetric matrix $\theta \in M_n(\R)$, let
\beu
C(\S^{2n-1}_\theta) := C^*(z_1, \ldots, z_n \hspace{3 pt} | \hspace{3 pt} z_k \text{ normal}, \hspace{3 pt} z_k z_j = e^{2 \pi i \theta_{jk}} z_j z_k, \hspace{3 pt} z_1 z_1^* + \ldots + z_n z_n^* = 1)
\eeu
denote the $\theta$-sphere of dimension $2n - 1$. If $\theta$ is such that $z_n$ is central, then for $\rho$ the top left $(n-1)\times(n-1)$ minor of $\theta$, let
\beu
C(\S^{2n-2}_\rho) := C(\S^{2n-1}_\theta)/\langle z_n - z_n^* \rangle
\eeu
denoted the $\rho$-sphere of dimension $2n - 2$. If $\theta$ is instead such that $z_n$ anticommutes with all of the other $z_i$, then let
\beu
\mathcal{R}^{2n-2}_\rho := C(\S^{2n-1}_\theta)/\langle z_n - z_n^* \rangle
\eeu
denote the \lq\lq twisted\rq\rq\hspace{0pt} analogue of the $\rho$-sphere of dimension $2n -2$. Each of the above objects admits an \textit{antipodal action} of $\Z/2\Z$, denoted $\alpha$, which negates each generator in the presentation. For appropriate choices of $\rho$ and $\omega$, we have that
\beu
C(\S^{2n-2}_\rho) \cong C(\S^{2n-3}_\omega) \circledast C(\Z/2\Z) \cong J(C(\S^{2n-3}_\omega), \text{triv})
\eeu
and
\beu
\mathcal{R}^{2n-2}_\rho \cong J( C(\S^{2n-3}_\omega), \alpha).
\eeu
In both cases, the antipodal action on the $(2n-3)$-dimensional $\omega$-sphere induces the antipodal action on the (twisted) join, in the usual way.
\end{defn} 

The antipodal action $\alpha$ is free, and both the trivial action and the antipodal action satisfy condition 2 of Question \ref{ques:thequestion}.

\begin{exam}(Consequences of \cite[Corollary 4.12]{self16thetasphere} and \cite[Theorem 1.8]{self17antisphere})\label{exam:spherestuff}
Let $A$ be an algebra in Definition \ref{defn:spheres} of dimension $k$, and let $B$ be an algebra in Definition \ref{defn:spheres} of dimension $k+1$. If $\Z/2\Z$ acts on $A$ and $B$ via the antipodal action, then there is no equivariant, unital $*$-homomorphism from $A$ to $B$.

Let $C$ be another algebra in Definition \ref{defn:spheres} of dimension $k$, and suppose $\phi: A \to C$ is equivariant. If $k$ is odd, then the $K_1$-groups of both algebras are isomorphic to $\Z$, and the induced map $\phi_*$ on $K_1$ is nontrivial. If $k$ is even, then the $K_0$-groups of both algebras are isomorphic to $\Z \oplus \Z$, with the first component generated by the unit $1$, and the induced map $\phi_*$ on $K_0$ is such that $\textrm{Ran}(\phi_*)$ is not cyclic.
\end{exam}

\section{Existence}\label{sec:existence}

Recall the following assumptions of Question \ref{ques:thequestion} for commuting $\Z/k\Z$-actions $\alpha$ and $\beta$ on a unital $C^*$-algebra $A$, where $\alpha$ is free. 

\begin{enumerate}
\item The action $\beta$ is not free.
\item The individual automorphisms of $\beta$ are connected within $\mathrm{Aut}(A)$ to the trivial automorphism.
\end{enumerate}

When we discuss continuous paths in $\textrm{Aut}(A)$ or $\textrm{Hom}(A, B)$, we will always mean continuous with respect to the pointwise norm topology. Both conditions assert that $\beta$ is in some sense similar to the trivial action, and the conditions are motivated by examples and counterexamples from \cite{selfsaturatedPUBLISHED, self17antisphere} and the previous section. In particular, condition 2 may be thought of as the demand that $\beta$ is \lq\lq orientation-preserving.\rq\rq\hspace{0pt} However, we find that neither condition is sufficient to rule out $(\alpha, \widetilde{\alpha})$-equivariant maps $A \to J(A, \beta)$.

\begin{thm}\label{thm:condition1fails}
Let $A = C(\mathbb{S}^1)$ be generated by the coordinate unitary $z$, and equip $A$ with the antipodal action $\alpha: z \mapsto -z$ and the conjugation action $\beta: z \mapsto z^*$ of $\Z/2\Z$. There is an $(\alpha, \widetilde{\alpha})$-equivariant, unital $*$-homomorphism from $A$ to $J(A, \beta)$. Since $\alpha$ is free and $\beta$ is not free, condition 1 of Question \ref{ques:thequestion} is insufficient.
\end{thm}
\begin{proof}
Let $z = x + iy$, so that $\alpha$ negates both $x$ and $y$, but $\beta$ fixes $x$ and negates $y$. Also, let $C^*(\Z/2\Z)$ be generated by the self-adjoint unitary $\mu$. It follows that in $A \rtimes_\beta \Z/2\Z$, $y \mu = - \mu y$ and $x \mu = \mu x$. The points $a_t, b_t \in A \rtimes_\beta \Z/2\Z$ defined by
\beu
a_t = tx, \hspace{1 in} b_t = ty + \sqrt{1 - t^2} \mu,
\eeu
for $t \in [0,1]$ are self-adjoint, commute with each other, and satisfy
\beu \ba
a_t^2 + b_t^2 	&= t^2 x^2 + \left(t^2 y^2 + t\sqrt{1-t^2}y\delta + t \sqrt{1 - t^2} \delta y + (1-t^2) \right)\\
					&= t^2(x^2 + y^2) + (1 - t^2) \\
					&= 1.
\ea \eeu
Since $a_0 = 0$ and $b_0 = \mu$ belong to $C^*(\Z/2\Z)$, and similarly $a_1 = x$ and $b_1 = y$ belong to $A$, it follows that $f(t) = a_t + i b_t$ is a unitary element in $J(A, \beta)$. Further, $\widetilde{\alpha}(f) = -f$, since $\widetilde{\alpha}$ is defined as the pointwise application of $\alpha \widehat{\beta}$, which negates $x$, $y$, and $\mu$ in $A \rtimes_\beta \Z/2\Z$. Finally, the unital $*$-homomorphism $\phi: A \to J(A, \beta)$ defined by $\phi(z) = f$ is $(\alpha, \widetilde{\alpha})$-equivariant.
\end{proof}

Next, consider the universal $C^*$-algebra 
\beu
A \cong C^*( x, y \hspace{2pt}|\hspace{2pt} x = x^*, y = y^*, x^2 + y^2 = 1).
\eeu
While $A$ is itself noncommutative, there is an obvious surjection from $A$ onto $C(\S^1)$. Moreover, $A$ admits a $\Z/2\Z$ action generated by 
\be\label{eq:antipodal_ns} 
\alpha: \begin{array}{ccc} x & \mapsto & -x \\ y & \mapsto & -y \end{array},
\ee
analogous to the antipodal action on the quotient $C(\S^1)$. Using a rotation argument motivated by the commutative quotient, we find that the automorphism which generates $\alpha$ is connected within $\mathrm{Aut}(A)$ to the trivial automorphism. Specifically, if $s, t \in \R$ have $s^2 + t^2 = 1$, then 
\beu \ba
(sx + ty)^2 + (-tx + sy)^2 &= (s^2 x^2 + st xy + st yx + t^2 y^2) + (t^2 x^2 - st xy - st yx + s^2 y^2) \\
								&= (s^2 + t^2)(x^2 + y^2) \\
								&= 1,
\ea \eeu
so there is an endomorphism $R_{s,t}: A \to A$ defined by $R_{(s,t)}(x) = sx + ty$ and $R_{(s,t)}(y) = -tx + sy$. The inverse of $R_{(s,t)}$ is $R_{(s,-t)}$, as
\beu
s(sx - ty) + t(tx + sy) = (s^2 + t^2)x = x,
\eeu
\beu
-t(sx - ty) + s(tx + sy) = (s^2 + t^2)y = y,
\eeu
and similarly for the reverse composition. Therefore, $R_{(s,t)}$ is an automorphism for each $(s, t) \in \S^1$. Finally, $\alpha_1 = R_{(-1,0)}$ is connected via a path $R_{(s,t)} \in \text{Aut}(A)$ to the trivial automorphism $R_{(1,0)}$.

\begin{thm}\label{thm:condition2fails}
Let $A = C^*(x, y \hspace{2pt}|\hspace{2pt} x = x^*, y = y^*, x^2 + y^2 = 1)$, equip $A$ with the action $\alpha$ of $\Z/2\Z$ that negates $x$ and $y$, and let $\alpha = \beta$. Then there is an $(\alpha, \widetilde{\alpha})$-equivariant, unital $*$-homomorphism from $A$ to $J(A, \beta)$. Since $\alpha$ is free and the automorphism generating $\beta$ is connected within $\text{Aut}(A)$ to the trivial automorphism, condition 2 of Question \ref{ques:thequestion} is insufficient.
\end{thm}
\begin{proof}
Let $C^*(\Z/2\Z)$ be generated by the self-adjoint unitary $\mu$, and define the self-adjoint elements $a_t, b_t \in A \rtimes_\beta \Z/2\Z$ by
\beu
a_t = tx + \cfrac{\sqrt{1-t^2}}{\sqrt{2}}\hspace{1pt} \mu, \hspace{1 in} b_t = ty + \cfrac{\sqrt{1-t^2}}{\sqrt{2}} \hspace{1pt} \mu.
\eeu
Since $\mu$ anticommutes with both $x$ and $y$, it follows that 
\beu \ba
a_t^2 + b_t^2 &= \left(t^2 x^2 + \cfrac{1-t^2}{2}\right) + \left(t^2 y^2 + \cfrac{1 - t^2}{2} \right) \\
				   &= t^2(x^2 + y^2) + (1 -t^2) \\
				   &= 1.
\ea \eeu
We also have that $a_0 = \frac{1}{\sqrt{2}} \mu = b_0 \in C^*(\Z/2\Z)$, and similarly, $a_1 = x$ and $b_1 = y$ belong to $A$. Therefore $f(t) = a_t$ and $g(t) = b_t$ define elements $f, g \in J(A, \beta)$. Since $f$ and $g$ are negated by $\widetilde{\alpha}$, the map $\phi: A \to J(A, \beta)$ defined by $\phi(x) = f$, $\phi(y) = g$ is an $(\alpha, \widetilde{\alpha})$-equivariant, unital $*$-homomorphism.
\end{proof}

Next, we expand upon some consequences of the two previous theorems.

\begin{remark}\label{rem:remarkathon}
In both Theorem \ref{thm:condition1fails} and Theorem \ref{thm:condition2fails}, the chosen equivariant morphism $A \to J(A, \beta)$ is such that evaluation at the $t = 1$ endpoint of the twisted join produces the usual embedding $A \hookrightarrow A \rtimes_\beta \Z/2\Z$. Further, since evaluation at $t = 0$ produces a map $A \to C^*(\Z/2\Z)$, which has a commutative codomain, and the largest commutative quotient of $A$ is $C(\S^1)$, there is a factorization
\beu
A \to C(\S^1) \xrightarrow{\Lambda} C^*(\Z/2\Z).
\eeu
The morphism $\Lambda$ is dual to a continuous function $\lambda: \Z/2\Z \to \S^1$, i.e. a selection of two points in $\S^1$. Since $\S^1$ is path connected, we may apply a path $\lambda_t$ connecting $\lambda = \lambda_0$ to $\lambda_1$, which selects two identical points. This produces a path connecting $\Lambda$ to a one-dimensional representation $A \to \C$. We conclude that the inclusion map $A \hookrightarrow A \rtimes_\beta \Z/2\Z$ may be connected to a one-dimensional representation.
\end{remark}

\begin{remark} Let $A$ be as in Theorem \ref{thm:condition2fails}, so $A$ again has $C(\S^1)$ as its largest commutative quotient. The commutator ideal of $A$ is invariant, and the induced action on the quotient $C(\S^1)$ is the antipodal action. However, if $\alpha = \beta$ is the antipodal action on $C(\S^1)$, then there is no $(\alpha, \widetilde{\alpha})$-equivariant unital $*$-homomorphism $C(\S^1) \to J(C(\S^1), \beta)$. In particular, $J(C(\S^1), \beta)$ is a twisted analogue of a $2$-sphere that appears in Definition \ref{defn:spheres} and Example \ref{exam:spherestuff}. Therefore, the failure of Borsuk-Ulam theorems is not preserved in quotient procedures.
\end{remark}

In the non-twisted setting, if a compact group $G$ acts on $A$, then existence of an equivariant morphism $A \to A \circledast C(G)$ is equivalent to the existence of a path in $\text{Hom}(A, A)$ between an equivariant endomorphism and a one-dimensional representation of $A$ (see, e.g., the more general result \cite[Lemma 2.5]{dab16revisitedARXIV}). Since the identity endomorphism is certainly equivariant for any action, it is of perhaps independent interest whether or not $\text{id}_A$ may be connected to a one-dimensional representation, regardless of the presence of an action. If such a path does exist, then $A$ is called \textit{unitally contractible}, as in \cite[Definition 2.6]{dab16revisitedARXIV}. In the commutative case $A = C(X)$, this property corresponds to contractibility of $X$. Moreover, it is immediate that if equivariant maps $A \to A \circledast C(G)$ cannot exist, then $A$ cannot be unitally contractible. To consider analogous concepts in the twisted setting, we extend the notion of unital contractibility \lq\lq modulo $k$\rq\rq\hspace{0pt} by consideration of finite-dimensional representations.

\begin{defn}
Let $A$ be a unital $C^*$-algebra $A$. Then $A$ is called \textit{unitally contractible modulo k} if the embedding $a \in A \mapsto a \otimes I_k \in A \otimes M_k(\C)$ is connected within $\mathrm{Hom}(A, A \otimes M_k(\C))$ to a $k$-dimensional representation, that is, a map $\rho \in \mathrm{Hom}(A, \C \otimes M_k(\C))$. Similarly, $A$ is called \textit{strongly unitally contractible modulo k} if the embedding $a \mapsto a \otimes I_k$ may be connected within $\mathrm{Hom}(A, A \otimes M_k(\C))$ to a one-dimensional representation $\rho \in \mathrm{Hom}(A, \C \otimes \C)$.
\end{defn}

\begin{exam}
Fix $k \geq 2$. The matrix algebra $M_k(\C)$ admits a free $\Z/k\Z$ action, and it is contractible modulo $k$, but it is not strongly unitally contractible modulo $k$.
\end{exam}
\begin{proof}
Let $V$ be diagonal with entries $1, \omega, \ldots, \omega^{k-1}$, where $\omega$ is a primitive $k$th root of unity. Conjugation by $V$ is a free $\Z/k\Z$ action of $M_k(\C)$, as the matrix $W$ which acts on the standard basis of $\C^n$ by $e_i \mapsto e_{i+1}$ is unitary and in the $\omega$-isotypic subspace of the action. Further, the embeddings $\textrm{id} \otimes 1$ and $1 \otimes \textrm{id}$ of $M_k(\C)$ into $M_k(\C) \otimes M_k(\C) \cong M_{k^2}(\C)$ are conjugate. That is, there is a unitary $U \in M_{k^2}(\C)$ such that for each $M \in M_k(\C)$, $U (M \otimes 1) U^* = 1 \otimes M$. Because the unitary group of $M_{k^2}(\C)$ is path connected, the two embeddings are connected via $M \mapsto U_t (M \otimes 1) U_t^*$, where $U_0 = I$, $U_1 = U$, and each $U_t$ is unitary. Therefore, $M_k(\C)$ is unitally contractible modulo $k$. However, $M_k(\C)$ has no one-dimensional representations, so it cannot be strongly unitally contractible modulo $k$.
\end{proof}

In analogy with \cite[Corollary 2.7]{dab16revisitedARXIV}, we seek a connection between Borsuk-Ulam theorems and (strong) unital contractibility of $A$ modulo the same $k$. Certainly, if $A$ has $K$-theory invariants which remain nontrivial under the quotient $G / kG$, then $A$ is not unitally contractible modulo $k$.

\begin{exam}
The circle $C(\S^1)$ is not unitally contractible modulo $k$ for any $k$. This holds even though Theorem \ref{thm:condition1fails} produces an equivariant map $C(\S^1) \to J(C(\S^1), \beta)$ such that evaluation at $t = 1$ gives the usual inclusion $A \hookrightarrow A \rtimes_\beta \Z/2\Z$. In particular, it is crucial for this example that $\beta$ is orientation-reversing.
\end{exam}

We may, however, adapt the counterexample in Theorem \ref{thm:condition2fails} to show the $C^*$-algebra $A$ used therein is \textit{strongly} unitally contractible modulo $2$.

\begin{exam}
The $C^*$-algebra $A = C^*(x, y \hspace{2pt}|\hspace{2pt} x = x^*, y = y^*, x^2 + y^2 = 1)$ is strongly unitally contractible modulo $2$.
\end{exam}
\begin{proof}
It is known from Theorem \ref{thm:condition2fails} that if $\alpha = \beta$ is the action generated by $x \mapsto -x$, $y \mapsto -y$, then there is an $(\alpha, \widetilde{\alpha})$-equivariant map $\phi: A \to J(A, \alpha)$. An examination of the proof shows that $\text{ev}_{t=1}(\phi(a)) = a$ for each $a \in A$. Expanding the crossed product $A \rtimes_\alpha \Z/2\Z$ via $E: a_0 + a_1 \mu \mapsto \begin{pmatrix} a_0 & a_1 \\ \alpha(a_1) & \alpha(a_0) \end{pmatrix}$ then shows that the embedding
\beu
\psi_1: a \in A \mapsto \begin{pmatrix} a  \\ & \alpha(a) \end{pmatrix} \in A \otimes M_2(\C)
\eeu
may be connected via a path $\psi_t \in \text{Hom}(A, A \otimes M_2(\C))$ to a homomorphism $\psi_0$ such that $\textrm{Ran}(\psi_0) \subseteq E(C^*(\Z/2\Z))$.

Both endpoints of the above path must be adjusted to show strong unital contractibility modulo $2$. First, the automorphism generating the action $\alpha$ is connected in $\mathrm{Aut}(A)$ via a path $\alpha_t \in \mathrm{Aut}(A)$ to $\alpha_0 = \mathrm{id}$. Therefore
\beu
a \in A \mapsto \begin{pmatrix} a  \\ & \alpha_t(a) \end{pmatrix} \in A \otimes M_2(\C)
\eeu
may be used to connect $\psi_1$ to the diagonal embedding $A \hookrightarrow A \otimes M_2(\C)$. Second, $\psi_0$ maps $A$ maps to a subalgebra of $M_2(\C)$ isomorphic to $C^*(\Z/2\Z)$, which is commutative. Using the factorization technique of Remark \ref{rem:remarkathon}, we see that $\psi_0$ may be connected to a one-dimensional representation. Gluing all of the constructed paths together shows that $A$ is indeed strongly unitally contractible modulo $2$.
\end{proof}

Finally, we have seen that the nontriviality of equivariant maps proved in \cite[Corollary 2.4]{selfsaturatedPUBLISHED} for $\Z/k\Z$ actions is removed when crossed products replace tensor products, or when a matrix algebra is introduced, even under a fairly stringent assumption on the twist.

\section*{Acknowledgments}

I am very grateful to Orr Shalit and Baruch Solel for their support during my current postdoc at Technion-Israel Institute of Technology, and to the organizers of the Simons Semester at IMPAN.

\bibliography{references2}

\end{document}